\documentclass[a4paper,12pt,reqno]{amsart}

\usepackage[T1]{fontenc}
\usepackage[utf8]{inputenc}
\usepackage{lmodern}
\usepackage[english]{babel}

\usepackage[%
	left=2.5cm,       
	right=2.5cm,      
	top=3.5cm,        
	bottom=3.5cm,     
	heightrounded,    
	bindingoffset=0mm 
]{geometry}

\usepackage{amsmath}
\usepackage{amssymb}
\usepackage{mathtools}
\usepackage{mathrsfs}
\usepackage{soul}
\usepackage[normalem]{ulem}

\numberwithin{equation}{section}

\usepackage[
colorlinks=true,
urlcolor=teal,
citecolor=magenta,
linkcolor=teal,
breaklinks=true
]
{hyperref}

\usepackage[
capitalize, 
nameinlink,
noabbrev
]{cleveref}

\usepackage{bookmark}

\theoremstyle{plain}
	\newtheorem{theorem}{Theorem}[section]
	\newtheorem{lemma}[theorem]{Lemma}
	\newtheorem{proposition}[theorem]{Proposition}
	\newtheorem{corollary}[theorem]{Corollary}

\theoremstyle{definition}
	\newtheorem{definition}[theorem]{Definition}
	
	\newtheorem{remark}[theorem]{Remark}

\usepackage[%
initials,
msc-links
]{amsrefs}

\usepackage{tikz}
\usetikzlibrary{decorations.pathreplacing}

\usepackage[font=small,labelfont=bf]{caption}

\usepackage{float}

\usepackage{enumerate}
\usepackage{url}

\newcommand{\N}{\mathbb{N}}
\newcommand{\R}{\mathbb{R}}

\newcommand{\Haus}{\mathcal{H}}

\newcommand{\di}{\,\mathrm{d}}

\DeclarePairedDelimiter{\set}{\{}{\}}

\allowdisplaybreaks

\mathchardef\ordinarycolon\mathcode`\:
\mathcode`\:=\string"8000
\begingroup \catcode`\:=\active
  \gdef:{\mathrel{\mathop\ordinarycolon}}
\endgroup

\begin{document}

\title[On the monotonicity of non-local perimeter of convex bodies]{On the monotonicity of non-local perimeter \\ of convex bodies}

\author[F.~Giannetti]{Flavia Giannetti}
\address[F.~Giannetti]{Dipartimento di Matematica ed Applicazioni ``R. Caccioppoli'',
Università degli Studi di Napoli ``Federico II'', Via Cintia, 80126 Napoli, Italy}
\email{giannett@unina.it}

\author[G.~Stefani]{Giorgio Stefani}
\address[G.~Stefani]{Scuola Internazionale Superiore di Studi Avanzati (SISSA), via Bonomea 265, 34136 Trieste, Italy}
\email{gstefani@sissa.it {\normalfont or} giorgio.stefani.math@gmail.com}

\date{\today}



\thanks{
\textit{Acknowledgements}. 
The authors are members of INdAM--GNAMPA.
The first-named author is a member of UMI--TAA group and has received funding from INdAM under the INdAM--GNAMPA 2023 Project \textit{Su alcuni problemi di regolarità del Calcolo delle Variazioni con convessità degenere} (grant agreement No.\ CUP\_E53\-C22\-001\-930\-001) and from the FRA Project 2020 \textit{Regolarità per minimi di funzionali ampiamente degeneri} (Project code 000022).
The second-named author has received funding from INdAM under the INdAM--GNAMPA 2023 Project \textit{Problemi variazionali per funzionali e operatori non-locali} (grant agreement No.\ CUP\_E53\-C22\-001\-930\-001), and has received funding from the European Research Council (ERC) under the European Union’s Horizon 2020 research and innovation program (grant agreement No.~945655).
The present research started during a visit of the second-named author at the Dipartimento di Matematica e Applicazioni `Renato Caccioppoli' of the Università degli Studi di Napoli `Federico II'. The second-named author wishes to thank the department for its kind hospitality.
} 

\begin{abstract}
Under mild assumptions on  the kernel $K\ge0$, the non-local $K$-perimeter $P_K$ satisfies the monotonicity property on nested convex bodies, i.e., if $A\subset B\subset\R^n$ are two convex bodies, then $P_K(A)\le P_K(B)$.
In this note, we prove  quantitative lower bounds on  the difference of the $K$-perimeters of $A$ and $B$ in terms of their Hausdorff distance, provided that $K$ satisfies suitable symmetry properties.
\end{abstract}

\maketitle

\section{Introduction}

\subsection{Monotonicity property}

Let $n\ge 2$.
If $A\subset B\subset\R^n$ are two nested convex bodies, that is, compact convex sets with non-empty interior, then
\begin{equation}
\label{eq:archimede}
P(A)\le P(B).
\end{equation}
Here $P(E)=\mathscr H^{n-1}(\partial E)$ is the Euclidean perimeter of the convex body $E\subset\R^n$ and $\mathscr H^{d}$ is the $d$-dimensional Hausdorff measure, $d\in[0,n]$.

The monotonicity property~\eqref{eq:archimede} is well known and dates back to the ancient Greeks (Archimedes took it as a postulate in his work on the sphere and the cylinder~\cite{A04}*{p.~36}). 
Inequality~\eqref{eq:archimede} follows from the \emph{Cauchy formula} for the area surface of convex bodies~\cite{BF87}*{\S7},  the monotonicity property of \emph{mixed volumes}~\cite{BF87}*{\S8}, the Lipschitz property of the projection on a convex closed set~\cite{BFK95}*{Lem.~2.4} and, finally, from the fact that the perimeter is decreased by intersection with half-spaces~\cite{M12}*{Ex.~15.13}. 

The monotonicity property~\eqref{eq:archimede} also holds for the \emph{anisotropic} (\emph{Wulff}) \emph{perimeter}
\begin{equation}
\label{eq:wulff}
P_\Phi(E)=\int_{\partial E}\Phi(\nu_E(x))\di\mathscr H^{n-1}(x)
\end{equation}
of a convex body $E\subset\R^n$, where $\nu_E\colon\partial E\to\mathbb S^{n-1}$, $\mathbb S^{n-1}=\set*{x\in\R^n : |x|=1}$, is the inner unitary normal of~$E$ (defined for $\mathscr H^{n-1}$-a.e.\ $x\in\partial E$) and $\Phi\colon\R^n\to[0,+\infty]$ is positively $1$-homogeneous and convex (see~\cite{M12}*{Ch.~20}).
Similarly to~\eqref{eq:archimede}, the monotonicity property of~\eqref{eq:wulff} is a consequence of the anisotropic \emph{Cauchy formula}, of the monotonicity property of \emph{mixed volumes}~\cite{BF87}*{\S7,\S8}, and of the fact that~\eqref{eq:wulff} is decreased by intersection with half-spaces~\cite{M12}*{Rem.~20.3}.

The monotonicity property of perimeters has gained increasing attention in recent years, see~\cites{H23,M22,DLW22,B21,DV20,SS23} for some applications and  related results.
A current active line of research concerns quantitative formulations of the monotonicity property.
Lower bounds on the deficit $\delta(B,A)=P(B)-P(A)$ in terms of the \emph{Hausdorff distance} $h(A,B)$ between~$A$ and~$B$ (see~\cite{S14}*{Sec.~1.8} for the definition) have been obtained for $n=2,3$ in~\cites{LL08,CGLP15,CGLP16} (also see the survey~\cite{G17}) and for any $n\ge2$ in~\cite{S18}. 
Actually, the main result of~\cite{S18} establishes a quantitative lower bound also for~\eqref{eq:wulff}, provided that~$\Phi$ possess suitable symmetry properties.
We stress that the inequalities proved in~\cites{LL08,CGLP15,CGLP16,S18} are sharp, in the sense that they are equalities at least in one non-trivial case.
Quantitative monotonicity inequalities for the perimeter can be applied to achieve lower bounds on the minimal number of convex components of a non-convex set~\cites{LL08,CGLP19,GS23}.

\subsection{Main results}

The aim of the present note is to study the monotonicity property of non-local perimeter functionals, also in its quantitative form.

Given a non-negative measurable \emph{kernel} $K\colon\R^n\to[0,+\infty]$, the associated \emph{non-local $K$-perimeter} of a measurable set $E\subset\R^n$ is defined as
\begin{equation}
\label{eq:K-perimeter}
P_K(E)=\frac12\int_{\R^n}\int_{\R^n}|\chi_E(x)-\chi_E(y)|\,K(x-y)\di x\di y.
\end{equation}
A prominent example of such non-local functional is the \emph{$s$-fractional perimeter}, $s\in(0,1)$,
\begin{equation}
\label{eq:s-perimeter}
P_s(E)=\int_{E}\int_{E^c}\frac{\di x\di y}{|x-y|^{n+s}},
\end{equation} 
induced by the \emph{fractional kernel} $K_s=|\cdot|^{-n-s}$, see~\cites{DPV12}.
The $K$-perimeter~\eqref{eq:K-perimeter} has attracted considerable attention in recent years,
see~\cites{C20,CSV19,P20,CN18,L14,K21,DNP21,CN22,CMP15,BP19,FPSS22,BS22}.

Due to the definition in~\eqref{eq:K-perimeter}, it is not restrictive to assume that
\begin{equation}
\label{ass:sym}
K(-x)=K(x)
\quad 
\text{for a.e.}\ x\in\R^n.
\end{equation}
Moreover, since we need the $K$-perimeter to be finite on convex bodies, we assume that 
\begin{equation}
\label{ass:nts}
\int_{\R^n}\min\set*{1,|x|}\,K(x)\di x<+\infty.
\end{equation}
Actually, condition~\eqref{ass:nts} yields that $P_K(E)<+\infty$ whenever $P(E)<+\infty$ and $|E|<+\infty$, see \cref{res:interpolation} below (where $|E|=\mathscr H^n(E)$ denotes the Lebesgue measure of $E\subset\R^n$).
We tacitly assume~\eqref{ass:sym} and~\eqref{ass:nts} throughout the paper. 

The monotonicity property of non-local perimeters was first proved for~\eqref{eq:s-perimeter} in~\cite{FFMMM15}*{Lem.~B.1} and then for~\eqref{eq:K-perimeter} in~\cite{BS22}*{Cor.~2.40} (we also refer to~\cite{CN22}*{Rem.~2.4} for the monotonicity of the localized functional $P_K(\,\cdot\,;B_R)$ for $R>0$, see~\eqref{eq:K-perimeter_localized} below for the precise definition).
Actually, the monotonicity results proved in~\cites{FFMMM15,CN22,BS22} require further assumptions on~$K$ in addition to~\eqref{ass:sym} and~\eqref{ass:nts}.
This is due to the fact that~\cites{FFMMM15,CN22,BS22} employ the monotonicity property of $P_K$ for convex sets which may be unbounded.

Our first main result is the following theorem, showing that assumptions~\eqref{ass:sym} and~\eqref{ass:nts} guarantee the monotonicity property of~$P_K$ on nested convex bodies. 

\begin{theorem}[$K$-monotonicity]
\label{res:monotonicity}
Let $n\ge2$. If $A\subset B\subset\R^n$ are two convex bodies, then 
\begin{equation}
\label{eq:monotonicity}
P_K(A)\le P_K(B).    
\end{equation}
\end{theorem}

The proof of \cref{res:monotonicity} follows the same strategy of the corresponding results in~\cites{FFMMM15,CN22,BS22}, but with some simplifications due to the boundedness on the involved sets.
Although the proof of~\eqref{eq:monotonicity} may be known to experts, we briefly outline it below to keep the present note as self-contained as possible.  

Our second main aim is to provide a quantitative version of \cref{res:monotonicity} by proving a lower bound on the non-local deficit 
$\delta_K(B,A)=P_K(B)-P_K(A)$ in terms of the Hausdorff distance $h(A,B)$ between $A$ and $B$. 
Our strategy is modeled on the approach adopted in~\cite{S18} for dealing with~\eqref{eq:wulff} and, essentially, requires that the kernel $K$ possess suitable symmetry properties.
Precisely, we assume that $x\mapsto K(x)$ is symmetric-decreasing with respect to the component of $x\in\R^n$ which is orthogonal to some fixed direction $\nu\in\mathbb S^{n-1}$. 

\begin{definition}[$\nu^\perp$-symmetric decreasing kernel]
\label{def:nu_sym}
Let $n\ge2$.
We say that the kernel~$K$ is \emph{$\nu^\perp$-symmetric decreasing} for some $\nu\in\mathbb S^{n-1}$ if there exists a measurable function 
$k_\nu\colon[0,+\infty)^2\to[0,+\infty]$
such that 
$r\mapsto k_\nu(r,t)$ is decreasing for all $t\ge0$ and 
\begin{equation*}
K(x)=k_\nu(|x-(x\cdot\nu)\nu|,|x\cdot\nu|)
\quad
\text{for $x\in\R^n$}.
\end{equation*}
\end{definition}

If $K$ is $\nu^\perp$-symmetric decreasing for some $\nu\in\mathbb S^{n-1}$, then it also satisfies~\eqref{ass:sym} and, moreover, it is also $(-\nu)^\perp$-symmetric decreasing with $k_{-\nu}=k_\nu$.
In particular, \cref{def:nu_sym} applies to \emph{radially-symmetric decreasing} kernels, i.e.,
$K(x)=\phi(|x|)$, $x\in\R^n$, for some decreasing $\phi\colon[0,+\infty]\to[0,+\infty]$.
Indeed, in this case, one chooses
\begin{equation*}
k_\nu(r,t)=\phi\left(\sqrt{r^2+t^2}\right)
\quad
\text{for}\ r,t\ge0,
\end{equation*}
for any $\nu\in\mathbb S^{n-1}$.
Note that this holds for~$P_s$ in~\eqref{eq:s-perimeter}, with $\phi(r)=r^{-n-s}$ for $r>0$.

With this notation at disposal, our second main result reads as follows.
Here and in the rest of the paper, we let 
$\langle\nu\rangle=\set*{t\nu : t\in\R}\subset\R^n$.

\begin{theorem}[Quantitative $K$-monotonicity]
\label{res:quantitative}
Let $n\ge2$ and let $K$ be $\nu^\perp$-sym\-me\-tric decreasing for some $\nu\in\mathbb S^{n-1}$ as in \cref{def:nu_sym}.
There exists a function 
$f_{K,\,\nu}\colon[0,+\infty)^3\to[0,+\infty)$
with the following property.
If $A\subset B\subset\R^n$ are two convex bodies with Hausdorff distance $h(A,B)=|a-b|$ achieved for some $a\in A$ and $b\in B$ such that $a-b\in\langle\nu\rangle$, then 
\begin{equation}
\label{eq:quantitative}
P_K(A)
+
f_{K,\nu}\left(h(A,B),\mathscr H^{n-1}(B\cap\partial H),|B\cap H|\right)
\le 
P_K(B),
\end{equation} 
where $H=\set*{x\in\R^n : (b-a)\cdot(x-a)\le0}$.
\end{theorem}

The function $f_{K,\nu}$ in \cref{res:quantitative} is implicit, due to the fact that, contrary to the local case~\cite{S18}, the non-locality of the functional~\eqref{eq:K-perimeter} prevents us to   compute the lower bound on the $K$-perimeter deficit explicitly.  
Nonetheless, the proof of \cref{res:quantitative} yields a partial characterization of optimal configurations.
In particular, inequality~\eqref{eq:quantitative} is sharp, in the sense that it is an equality at least in one non-trivial case.

The strategy of the proof of \cref{res:quantitative} is inspired by~\cites{S18,K21}.
We first reduce the given convex bodies to more symmetric ones by performing a \emph{Schwartz symmetrization} (done orthogonally to the given $\nu\in\mathbb S^{n-1}$), thanks to the symmetry of $K$ ensured by \cref{def:nu_sym}, and then we exploit a compactness argument.

As a consequence of \cref{res:quantitative}, in the radially-symmetric decreasing case, we get \cref{res:phi-quantitative} below.
Here and in the following, we let $P_\phi=P_{\phi(|\cdot|)}$ for any decreasing  function $\phi\colon[0,+\infty]\to[0,+\infty]$, we let $t^+=\max\set*{t,0}$ be the positive part of $t\in\R$ and, given $h,r\ge0$, we let
\begin{equation*}
\mathcal C_r^h
=
\bigcup_{x\in D_r}
\set*{(1-t)h\mathrm e_n+tx:t\in[0,1]}
\end{equation*}
be the (closed) right circular cone with base 
$D_r=\set*{x\in\R^n : |x|\le r,\ x_n=0}$ and height~$h$ in the direction $\mathrm e_n=(0,\dots,0,1)\in\mathbb S^{n-1}$.
Finally, we let $h=h(A,B)$ denote the Hausdorff distance between $A$ and~$B$ and, as usual, $\omega_{n-1}=\mathscr H^{n-1}(\mathbb S^{n-1})$ .

\begin{corollary}[Quantitative $\phi$-monotonicity]
\label{res:phi-quantitative}
Let $n\ge2$. If $A\subset B\subset\R^n$ are two convex bodies, then
\begin{equation}
\label{eq:phi-quantitative}
P_\phi(A)
+
\left(
P_\phi(\mathcal C_r^h)
-
\sigma_\phi\max\set*{\tfrac12P(B\cap H),|B\cap H|}
\right)^+
\le 
P_\phi(B)
\end{equation}
where $\sigma_\phi=\int_{\R^n}\min\set*{1,|x|}\,\phi(|x|)\di x$,
\begin{equation*}
r=\sqrt[n-1]{\frac{\mathscr H^{n-1}(B\cap\partial H)}{\omega_{n-1}}},
\quad
H=\set*{x\in\R^n : (b-a)\cdot(x-a)\le0},
\end{equation*}
and $a\in A$ and $b\in B$ such that $|a-b|=h(A,B)$.
\end{corollary}

Inequality~\eqref{eq:phi-quantitative} is apparently worse than~\eqref{eq:quantitative}.
However, the lower bound given by~\eqref{eq:phi-quantitative} is more explicit than the one given by~\eqref{eq:quantitative}.
In fact, one is only left to estimate $P_\phi(\mathcal C^h_r)$, which may be explicitly done for a specific choice of~$\phi$.

A special instance of \cref{res:phi-quantitative} is the fractional case.
In fact, exploiting the \emph{fractional isoperimetric inequality} (see~\cite{FFMMM15} and the references therein) 
\begin{equation}
\label{eq:s-iso}
P_s(E)\ge c^{\rm iso}_{n,s}\,|E|^{\frac{n-s}{n}},
\quad 
c^{\rm iso}_{n,s}
=
\frac{P_s(B_1)}{|B_1|^{\frac{n-s}{n}}},
\end{equation}
valid for any measurable $E\subset\R^n$, we get the following result.
  
\begin{corollary}[Quantitative $s$-monotonicity]
\label{res:fractional}
Let $n\ge2$.
If $A\subset B\subset\R^n$ are two convex bodies, then
\begin{equation}
\label{eq:s-quantitative}
P_s(A)
+
\left(
c^{\rm iso}_{n,s}
\left(\tfrac{\omega_{n-1}}{n}\,h\,r^{n-1}
\right)^{1-\frac sn}
-
\tfrac{2^{2-s}\,n\,\omega_n}{s\,(1-s)}\,P(B\cap H)^s\,|B\cap H|^{1-s}
\right)^+
\le 
P_s(B)
\end{equation}
where $h,r\ge0$, $a\in A$, $b\in B$ and $H\subset\R^n$ are as in \cref{res:phi-quantitative}.
\end{corollary}

Inequality~\eqref{eq:s-quantitative} is worse than~\eqref{eq:phi-quantitative}, since we used~\eqref{eq:s-iso} to estimate the term $P_s(\mathcal C_r^h)$ from below in terms of $|\mathcal C_r^h|$ which, in turn, can be explicitly computed from~$h,r$. 
Anyway, the lower bound we get is more explicit than the one given by  \eqref{eq:phi-quantitative}.

\subsection{The 1D case}

In closing, let us comment on the special case $n=1$. 

The monotonicity property~\eqref{eq:archimede} becomes trivial for $n=1$, since $P(A)=P(B)=2$ for any two (not necessarily nested) segments $A,B\subset\R$.
Instead,  the case $n=1$ becomes non-trivial in the non-local setting. 
For example, in the fractional case, due to the  scaling and translation invariance of~\eqref{eq:s-perimeter}, we have
\begin{equation}
\label{eq:s-1}
P_s(B)-P_s(A)
=
c_s\left(|B|^{1-s}-|A|^{1-s}\right)
\end{equation} 
for any two (not necessarily nested) segments $A,B\subset\R$, where $c_s=P_s((0,1))$.

With~\eqref{eq:s-1} in mind, we have the following  result, which is inspired by~\cite{BS22}*{Lem.~2.31}.

\begin{proposition}[Case $n=1$]
\label{res:segmenti}
Assume that $\phi\colon[0,+\infty)\to(0,+\infty)$ satisfies\begin{equation}
\label{eq:2-decreasing}
\inf_{t>0}\frac{R^2\phi(Rt)-r^2\phi(rt)}{\phi(t)}
\ge
\psi(R)-\psi(r)
\quad
\text{for all}\ R\ge r\ge0,  
\end{equation}
for some increasing function $\psi\colon[0,+\infty)\to[0,+\infty)$. 
If $A,B\subset\R$ are two segments with $|A|\le|B|$, then, setting $c_\phi=P_\phi((0,1))$, 
\begin{equation}
\label{eq:segmenti}
P_\phi(A)
+
c_\phi
\big(\psi(|B|)-\psi(|A|)\big)
\le 
P_\phi(B).
\end{equation}
\end{proposition}

Note that the assumption in~\eqref{eq:2-decreasing} implies that $\phi$ is decreasing.
In the fractional case, $\phi(r)=r^{-1-s}$ for $r>0$, so $\psi(r)=r^{1-s}$ for $r\ge0$, and we recover~\eqref{eq:s-1}.

\section{Proofs of the statements}

\label{sec:proofs}

\subsection{Proof of \texorpdfstring{\cref{res:monotonicity}}{non-local monotonicity}}

\label{subsec:monotonicity}

\cref{res:monotonicity} is a consequence of the following result.

\begin{proposition}[Intersection with convex sets]
\label{res:intersection}
If $E\subset\R^n$ is a convex body, then $P_K(E\cap C)\le P_K(E)$
for any convex  set $C\subset\R^n$.
\end{proposition}

The proof of \cref{res:intersection} exploits the \emph{local minimality} of half-spaces, see
\cref{res:local-minimality} below.
This latter result was proved first for~\eqref{eq:s-perimeter} in~\cite{ADM11}*{Prop.~17} and then for~\eqref{eq:K-perimeter} in~\cite{P20}*{Th.~1} (also see~\cite{C20}*{Cor.~2.5} and~\cite{BS22}*{Lem.~2.31}).
Here and below, for measurable sets $E,A\subset\R^n$, we let
\begin{equation}
\label{eq:K-perimeter_localized}
P_K(E;A)
=
\left(
\int_{E\cap A}\int_{E^c\cap A}
+
\int_{E\cap A}\int_{E^c\cap A^c}
+
\int_{E\cap A^c}\int_{E^c\cap A}
\right)\,K(x-y)\di x\di y
\end{equation}
be the \emph{$K$-perimeter of $E$ relative to~$A$}.

\begin{lemma}[Local minimality of half-spaces]
\label{res:local-minimality}
Let $R>0$.
If $H\subset\R^n$ is a half-space, then
$P_K(H;B_R)\le P_K(E;B_R)$ for any $E\subset\R^N$ such that $E\setminus B_R=H\setminus B_R$.
\end{lemma}

\begin{proof}[Proof of \cref{res:intersection}]
Our strategy is a simplification of the one used for the proof of~\cite{BS22}*{Th.~2.29} (see also~\cite{FFMMM15}*{Lem.~B.1} and~\cite{CN22}*{Rem.~2.4}). 
Since $C\subset\R^n$ is convex, we can find a sequence of half-spaces $(H_k)_{k\in\N}$ such that the sets 
\begin{equation*}
C_N=\bigcap\limits_{k=1}^N H_k,
\quad 
N\in\N,
\end{equation*}
satisfy
$|C_N\bigtriangleup C|\to0^+$ as $N\to+\infty$.
Therefore, thanks to the lower semicontinuity  property of~$P_K$ (see~\cite{BS22}*{Lem.~2.10} for instance), we can assume that $C=H$ is a half-space.
We can hence write
\begin{align*}
P_K(E)&-P_K(E\cap H)
=
\left(\int_{E}\int_{E^c}-\int_{E\cap H}\int_{(E\cap H)^c}\right)\,K(x-y)\di x\di y
\\
&=
\left(\int_{E\cap H}+\int_{E\setminus H}\right)\int_{E^c}\,K(x-y)\di x\di y
-
\left(\int_{E^c}+\int_{E\setminus H}\right)\int_{E\cap H}\,K(x-y)\di x\di y
\\
&=
\left(\int_{E^c}-\int_{E\cap H}\right)\int_{E\setminus H}\,K(x-y)\di x\di y.
\end{align*}
Now let $R>0$ be such that $E\subset B_R$.
Defining $F=E\cup H$, one easily checks that $E\subset F$, $E\setminus H=F\setminus H$, $F\cap H=H$ and $F\setminus B_R=H\setminus B_R$.
A plain computation then yields
\begin{align*}
\left(\int_{E^c}-\int_{E\cap H}\right)\int_{E\setminus H}\,K(x-y)\di x\di y
&\ge
\left(\int_{F^c}-\int_{F\cap H}\right)\int_{F\setminus H}\,K(x-y)\di x\di y.
\end{align*}
We now observe that the right-hand side of the above inequality can be rewritten as
$P_K(F;B_R)-P_K(H;B_R)$ 
exploiting the definition in~\eqref{eq:K-perimeter_localized}.
Therefore, thanks to \cref{res:local-minimality}, we get that
\begin{equation*}
P_K(E)-P_K(E\cap H)
\ge
P_K(F;B_R)-P_K(H;B_R)
\ge
0,
\end{equation*}
yielding the conclusion.
\end{proof}

\subsection{Proof of \texorpdfstring{\cref{res:quantitative}}{quantitative monotonicity}}

We begin by adapting~\cite{M12}*{Sec.~19.2} (which corresponds to the choice $\nu=\mathrm e_1$ in what follows) to our setting.

Let $\nu\in\mathbb S^{n-1}$ be fixed.
For any $x\in\R^n$, we set 
\begin{equation*}
x'_\nu=x-(x\cdot\nu)\nu
\quad
\text{and}
\quad
x_\nu=(x\cdot\nu)\nu.
\end{equation*}
We naturally identify $x_\nu\in\langle\nu\rangle$ and $x_\nu'\in\langle\nu\rangle^\perp$, where $\langle\nu\rangle^\perp$ is the linear space orthogonal to $\langle\nu\rangle$, with points in $\R$ and $\R^{n-1}$, respectively.
In particular, with a slight abuse of notation, we write $x=x_\nu'+x_\nu=(x'_\nu,x_\nu)$ for any $x\in\R^n$.

\begin{definition}[Schwartz $\nu$-symmetrization] 
Given $E\subset\R^n$ with $|E|<+\infty$, we let 
\begin{equation}
\label{eq:nu-slice}
E^\nu_t=\set*{y\in\langle\nu\rangle^\perp:y+t\nu\in E},
\quad
t\in\R,
\end{equation}
be the \emph{slice of $E$ orthogonal to $\nu\in\mathbb S^{n-1}$}.
We hence let 
\begin{equation}
\label{eq:schwartz}
E^{\#\nu}
=
\set*{x\in\R^n : \omega_{n-1}|x'_\nu|^{n-1}\le\mathscr H^{n-1}\left(E^\nu_{x_\nu}\right)}
\end{equation}
be the \emph{Schwartz $\nu$-symmetrization of $E$}.
\end{definition}

The set $E^{\#\nu}$ is measurable, with slice $(E^{\#\nu})^\nu_t$ equal to an open ball such that 
\begin{equation*}
\mathscr H^{n-1}((E^{\#\nu})^\nu_t)=\mathscr H^{n-1}(E^\nu_t)
\end{equation*}
for each $t\in\R$.
Hence $|E^{\#\nu}|=|E|$ by Fubini--Tonelli's Theorem.

The following result is a non-local analog of the Schwartz inequality $P(E)\ge P(E^{\#\nu})$ for the Eulidean perimeter (see~\cite{M12}*{Th.~19.11}).
Actually, inequality~\eqref{eq:schwartz_interaction} below is a special case of~\cite{BLL74}*{Lem.~3.2}, but we give a direct and simpler proof of it via the well-known \textit{Riesz rearrangement inequality} (see~\cite{LL01}*{Ch.~3} for a detailed presentation).
Here and in the rest of the paper, we let
\begin{equation*}
L_K(E,F)=\int_E\int_F K(x-y)\di x\di y
\end{equation*}
be the \emph{$K$-interaction functional} between the two measurable sets $E,F\subset\R^n$. 

\begin{lemma}[Non-local Schwartz $\nu$-inequality]
\label{res:schwarz} 
Let $n\ge2$ and let $K$ be $\nu^\perp$-symmetric decreasing kernel for some $\nu\in\mathbb S^{n-1}$ as in \cref{def:nu_sym}.
If $E,F\subset\R^n$ are such that $|E|,|F|<+\infty$, then
\begin{equation}
\label{eq:schwartz_interaction}
L_K(E,F)\le L_K(E^{\#\nu},F^{\#\nu}).    
\end{equation}
Moreover, if $P_K(E)<+\infty$, then also $P_K(E^{\#\nu})<+\infty$ with
\begin{equation}
\label{eq:schwartz_perimeter}
P_K(E)\ge P_K(E^{\#\nu}).    
\end{equation}
\end{lemma}

\begin{proof}
Let $k_\nu$ be the function given by \cref{def:nu_sym}.
By Tonelli's Theorem, and recalling the definition of slice in~\eqref{eq:nu-slice}, we can write
\begin{equation}
\label{eq:interaction_coarea}
L_K(E,F)
=
\int_\R\int_\R
L_{k_\nu(|\,\cdot\,|,|x_\nu-y_\nu|)}(E^\nu_{x_\nu},F^\nu_{y_\nu})
\di x_\nu\di y_\nu,
\end{equation}
where
\begin{equation}
\label{eq:interaction_nu}
L_{k_\nu(|\,\cdot\,|,|x_\nu-y_\nu|)}(E^\nu_{x_\nu},F^\nu_{y_\nu})
=
\int_{\R^{n-1}}\int_{\R^{n-1}}
\chi_{E^\nu_{x_\nu}}(x'_\nu)
\,
\chi_{F^\nu_{y_\nu}}(y'_\nu)
\,
k_\nu(|x_\nu'-y_\nu'|,|x_\nu-y_\nu|)
\di x_\nu'\di y_\nu'.
\end{equation}
Since $|E|,|F|<+\infty$, we also have $\mathscr H^{n-1}(E^\nu_{x_\nu}),\mathscr H^{n-1}(F^\nu_{y_\nu})<+\infty$ for a.e.\ $x_\nu,y_\nu\in\langle\nu\rangle$.
Since the function $z\mapsto k_\nu(|z|,|x_\nu-y_\nu|)$ is radially-symmetric decreasing by \cref{def:nu_sym}, by Riesz rearrangement inequality~\cite{LL01}*{Th.~3.4} 
we infer that
\begin{equation}
\label{eq:simmetrizzo}
L_{k_\nu(|\,\cdot\,|,|x_\nu-y_\nu|)}(E^\nu_{x_\nu},F^\nu_{y_\nu})
\le 
L_{k_\nu(|\,\cdot\,|,|x_\nu-y_\nu|)}
(
D[E_{x_\nu}^\nu]
,
D[F_{y_\nu}^\nu]
)
\quad
\text{for a.e.\ $x_\nu,y_\nu\in\langle\nu\rangle$}.
\end{equation}
Here $D[E_{x_\nu}^\nu]
$ and $D[F_{y_\nu}^\nu]$ are the closed $(n-1)$-dimensional discs in~$\langle\nu\rangle^\perp$ centered at the origin with $(n-1)$-dimensional volume equal to $\mathscr H^{n-1}(E^\nu_{x_\nu})$ and $\mathscr H^{n-1}(F^\nu_{y_\nu})$, respectively.
Integrating~\eqref{eq:simmetrizzo}, and again using Tonelli's Theorem and~\eqref{eq:schwartz} and~\eqref{eq:interaction_coarea}, we get~\eqref{eq:schwartz_interaction}.
Similarly, choosing $F=E$ in~\eqref{eq:interaction_coarea}, we can write
\begin{equation}
\label{eq:perimetro_rotto}
P_K(E)
=
\int_\R\int_\R
L_{k_\nu(|\,\cdot\,|,|x_\nu-y_\nu|)}(E^\nu_{x_\nu},(E^c)^\nu_{y_\nu})
\di x_\nu\di y_\nu.
\end{equation}
Since $P_K(E)<+\infty$, we must have
\begin{equation}
\label{eq:finitezza}
L_{k_\nu(|\,\cdot\,|,|x_\nu-y_\nu|)}(E^\nu_{x_\nu},(E^c)^\nu_{y_\nu})
<+\infty
\quad
\text{for a.e.\ $x_\nu,y_\nu\in\langle\nu\rangle$}.
\end{equation}
Now, for any fixed $x_\nu,y_\nu\in\langle\nu\rangle$, we can write 
\begin{equation*}
k_\nu\left(|z|,|x_\nu-y_\nu|\right)
=
\int_0^{+\infty}
\chi_{\set*{k_\nu(|\cdot|,|x_\nu-y_\nu|)>t}}(z)\di t,
\quad
\text{for}\
z\in\langle\nu\rangle^\perp.
\end{equation*}
Therefore, for a.e.\ $x_\nu,y_\nu\in\langle\nu\rangle$, we can decompose  
\begin{equation}
\label{eq:livelli}
L_{k_\nu(|\,\cdot\,|,|x_\nu-y_\nu|)}(E^\nu_{x_\nu},(E^c)^\nu_{y_\nu})
=
\int_0^{+\infty}
L_{\chi_{\set*{k_\nu(|\cdot|,|x_\nu-y_\nu|)>t}}}(E^\nu_{x_\nu},(E^c)^\nu_{y_\nu})
\di t,
\end{equation}
where, as in~\eqref{eq:interaction_nu}, we have
\begin{equation*}
\begin{split}
L_{\chi_{\set*{k_\nu(|\cdot|,|x_\nu-y_\nu|)>t}}}&(E^\nu_{x_\nu},(E^c)^\nu_{y_\nu})
\\
&=
\int_{\R^{n-1}}\int_{\R^{n-1}}
\chi_{E^\nu_{x_\nu}}(x'_\nu)
\,
\chi_{(E^c)^\nu_{y_\nu}}(y'_\nu)
\,
\chi_{\set*{k_\nu(|\cdot|,|x_\nu-y_\nu|)>t}}(x_\nu'-y_\nu')
\di x_\nu'\di y_\nu'.
\end{split}
\end{equation*}
We now observe that, for a.e.\ $x_\nu,y_\nu\in\langle\nu\rangle$ and for any $t>0$, the set 
\begin{equation*}
\set*{z\in\langle\nu\rangle^\perp : k_\nu(|z|,|x_\nu-y_\nu|)>t}
\end{equation*}
is an $(n-1)$-dimensional disc (possibly empty or the entire subspace $\langle\nu\rangle^\perp$), with
\begin{equation}
\label{eq:finitezza_t}
\mathscr H^{n-1}
\left(
\set*{z\in\langle\nu\rangle^\perp : k_\nu(|z|,|x_\nu-y_\nu|)>t}
\right)
<
+\infty
\end{equation}
for a.e.\ $t>0$. 
Indeed, if this is not the case, then, for some $S\subset\R$ with $\mathscr H^1(S)>0$,  
\begin{equation*}
\begin{split}
L_{\chi_{\set*{k_\nu(|\cdot|,|x_\nu-y_\nu|)>t}}}(E^\nu_{x_\nu},(E^c)^\nu_{y_\nu})
&=
\int_{\R^{n-1}}\int_{\R^{n-1}}
\chi_{E^\nu_{x_\nu}}(x'_\nu)
\,
\chi_{(E^c)^\nu_{y_\nu}}(y'_\nu)
\di x_\nu'\di y_\nu'
\\
&=
\mathscr H^{n-1}(E^\nu_{x_\nu})
\,
\mathscr H^{n-1}((E^c)^\nu_{y_\nu})
=+\infty
\end{split}
\end{equation*}
for all $t\in S$, which, together with~\eqref{eq:livelli}, contradicts~\eqref{eq:finitezza}.
Now, for any $t>0$ yielding~\eqref{eq:finitezza_t}, we have $\chi_{\set*{k_\nu(|\cdot|,|x_\nu-y_\nu|)>t}}\in L^1(\langle\nu\rangle^\perp,\mathscr H^{n-1})$.
Hence, we can decompose
\begin{equation*}
\begin{split}
&L_{\chi_{\set*{k_\nu(|\cdot|,|x_\nu-y_\nu|)>t}}}
(E^\nu_{x_\nu},(E^c)^\nu_{y_\nu})
=
L_{\chi_{\set*{k_\nu(|\cdot|,|x_\nu-y_\nu|)>t}}}(E^\nu_{x_\nu},\langle\nu\rangle^\perp)
-
L_{\chi_{\set*{k_\nu(|\cdot|,|x_\nu-y_\nu|)>t}}}(E^\nu_{x_\nu},E^\nu_{y_\nu})
\\
&\qquad=
\mathscr H^{n-1}(E^\nu_{x_\nu})
\,
\mathscr H^{n-1}(\set*{k_\nu(|\cdot|,|x_\nu-y_\nu|)>t})
-
L_{\chi_{\set*{k_\nu(|\cdot|,|x_\nu-y_\nu|)>t}}}(E^\nu_{x_\nu},E^\nu_{y_\nu}).
\end{split}
\end{equation*}
Since 
$\mathscr H^{n-1}(E^\nu_{x_\nu})
=
\mathscr H^{n-1}(D[E_{x_\nu}^\nu])$
by~\eqref{eq:schwartz}
and
\begin{equation*}
L_{\chi_{\set*{k_\nu(|\cdot|,|x_\nu-y_\nu|)>t}}}(E^\nu_{x_\nu},E^\nu_{y_\nu})
\le 
L_{\chi_{\set*{k_\nu(|\cdot|,|x_\nu-y_\nu|)>t}}}(
D[E_{x_\nu}^\nu]
,
D[E_{y_\nu}^\nu])
\end{equation*}
by~\eqref{eq:simmetrizzo} (applied with $E=F$), again recalling~\eqref{eq:schwartz} we readily get that 
\begin{equation*}
L_{\chi_{\set*{k_\nu(|\cdot|,|x_\nu-y_\nu|)>t}}}
(E^\nu_{x_\nu},(E^c)^\nu_{y_\nu})
\ge 
L_{\chi_{\set*{k_\nu(|\cdot|,|x_\nu-y_\nu|)>t}}}((E^{\#\nu})^\nu_{x_\nu},((E^{\#\nu})^c)^\nu_{y_\nu}).
\end{equation*}
Integrating back in $t>0$ and recalling~\eqref{eq:livelli}, we get
\begin{equation*}
L_{\chi_{\set*{k_\nu(|\cdot|,|x_\nu-y_\nu|)>t}}}(E^\nu_{x_\nu},(E^c)^\nu_{y_\nu})
\ge 
L_{\chi_{\set*{k_\nu(|\cdot|,|x_\nu-y_\nu|)>t}}}
((E^{\#\nu})^\nu_{x_\nu},((E^{\#\nu})^c)^\nu_{y_\nu}).
\end{equation*}
Finally, integrating back in $x_\nu,y_\nu\in\langle\nu\rangle$ and recalling~\eqref{eq:perimetro_rotto}, we get~\eqref{eq:schwartz_perimeter}.
\end{proof}

In the proof of \cref{res:quantitative}, we will use the following notation.
Given $p\in\R^n$ and a (non-empty) set $S\subset\R^n$, we define the cones with vertex $p$ and base $S$
\begin{equation*}
\mathcal C(p,S)
=
\bigcup_{s\in S}\set*{p+t(s-p) : t\in[0,1]},
\quad
\mathcal C_\infty(p,S)
=
\bigcup_{s\in S}\set*{p+t(s-p) : t\ge0}.
\end{equation*}
Note that, if $S$ is convex, then also $\mathcal C(p,S)$ and $\mathcal C_\infty(p,S)$ are convex.
Moreover, if $S$ is bounded, then also $\mathcal C(p,S)$ is bounded.
Finally, given $p\in\R^n$, $r\ge0$ and $\nu\in\mathbb S^{n-1}$, we let 
\begin{equation*}
D_r^\nu(p)=p+D^\nu_r(0),
\quad
D_r^\nu(0)=
\set*{x\in\R^n : x\in\langle\nu\rangle^\perp,|x|\le r},
\end{equation*} 
be the closed $(n-1)$-dimensional disc centered at $p$, with radius $r$, and orthogonal to $\nu$.  

\begin{proof}[Proof of \cref{res:quantitative}]
Let $\nu\in\mathbb S^{n-1}$, $a\in A$, $b\in B$, $h=h(A,B)=|a-b|$ and $H\subset\R^n$ be as in the statement.
Since $A\subset B$ are compact convex sets, we have that
\begin{equation*}
h(A,B)=|a-b|=\max_{y\in B}\min_{x\in  A} |x-y|.
\end{equation*}
Consequently, $a$ is the orthogonal projection of~$b$ on~$A$. 
By definition of~$H$ and by minimality of the projection, the closed hyperplane~$\partial H$ is a supporting one for~$A$ in~$a$.
As a consequence, we must have that $A\subset B\cap H$.

\begin{figure}
\centering
\begin{tikzpicture}[scale=.6]
\fill [teal, opacity=.3] (0,0)--(0,5.5)--(5,5.95)--(5,4)--(1,5)--(1,2)--(5,1)--(5,0);
\fill [teal, opacity=.3] (5,0)--(6,0)--(10,2);
\fill [teal, opacity=.3] (6,6)--(10,2)--(5,5.95);
\fill [magenta, opacity=.3] (1,2)--(1,5)--(5,4)--(5,1)--(1,2);
\fill [gray, opacity=.2] (5,5.95)--(10,2)--(5,0)--(5,6);
\draw (0,0)--(0,5.5)--(6,6)--(10,2)--(6,0)--(0,0);
\draw (1,2)--(1,5)--(5,4)--(5,1)--(1,2);
\draw [dashed] (5,-1)--(5,7);
\draw [dashed] (5,5.95)--(10,2)--(5,0);
\draw (3,3) node {$A$};
\draw (1,1) node {$B$};
\draw (6,4) node {$C$};
\draw (5,6.5) node [left] {$\partial H$};
\draw (5,2) node [above left] {$a$};
\draw (5,2) node {\tiny\textbullet};
\draw [->,dashed] (-.5,2)--(11,2);
\draw (10,2) node [above right] {$b$};
\draw (10,2) node {\tiny\textbullet};
\draw (11,2) node [right] {$\nu$};
\draw [decorate,decoration={brace,amplitude=5pt},xshift=0pt,yshift=5pt] (5,2) -- (10,2) node [black,midway,yshift=12pt] {$h$};
\draw (1.5,5.75) node [above] {$E=B\cap H$};
\begin{scope}[shift={(14,-1)}]
\fill [gray, opacity=.2] (5,6)--(10,3)--(5,0)--(5,6);
\fill [teal, opacity=.3] (0,0.5)--(0,5.5)--(5,6)--(5,0)--(0,0.5);
\draw (0,0.5)--(0,5.5)--(5,6)--(10,3)--(5,0)--(0,0.5);
\draw [->,dashed] (-0.5,3)--(11,3);
\draw [dashed] (5,-1)--(5,7);
\draw (5,0)--(5,6);
\draw (11,3) node [right] {$\nu$};
\draw (6,4.5) node {$C^{\#\nu}$};
\draw (2,4.5) node {$E^{\#\nu}$};
\draw (5,3) node [above left] {$a$};
\draw (5,3) node {\tiny\textbullet};
\draw (10,3) node [above right] {$b$};
\draw (10,3) node {\tiny\textbullet};
\draw [decorate,decoration={brace,amplitude=5pt},xshift=0pt,yshift=5pt] (5,3) -- (10,3) node [black,midway,yshift=12pt] {$h$};
\draw [decorate,decoration={brace,amplitude=5pt,mirror},xshift=-2pt,yshift=0pt] (5,3) -- (5,0) node [black,midway,xshift=-12pt] {$r$};
\draw (5,6.5) node [left] {$\partial H$};
\end{scope}
\end{tikzpicture}
\caption{Reduction to symmetric sets in the proof of \cref{res:quantitative}: an initial configuration (left) and its symmetrization (right).}
\label{fig:sym}
\end{figure}
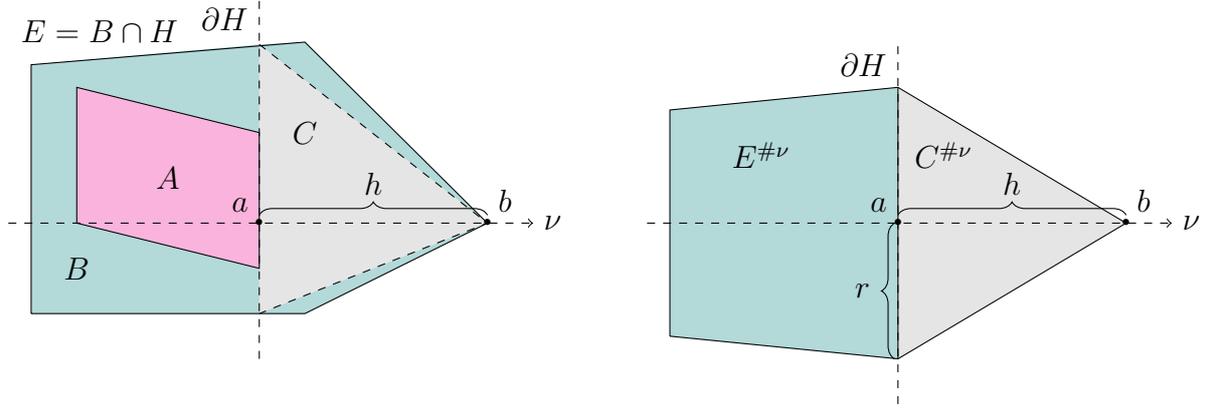

\vspace{1ex}

\textit{Step~1: reduction to symmetric sets}.
Let $E=B\cap H$ and $C=\mathcal C(b,B\cap \partial H)$ be the (bounded and closed) cone with vertex~$b$ and base $B\cap\partial H$.
Note that $E$, $C$ and $E\cup C$ are convex bodies.
Since $A\subset E$ and $E\cup C\subset B$, we have
\begin{equation*}
\delta_K(B,A)
\ge
\delta_K(B,E)
\ge
\delta_K(E\cup C,E),
\end{equation*}
with equality if $A=E$ and $B=E\cup C$.
Since $|E\cap C|=0$ and $E^c=C\cup(E^c\setminus C)=C\cup(E^c\cap C^c)$, we can write
\begin{equation}
\label{eq:insiemistica}
\begin{split}
P_K(E&\cup C)-P_K(E)
=
L_K(E\cup C,E^c\cap C^c)
-
L_K(E,E^c)
\\
&=
L_K(E,E^c\cap C^c)
+
L_K(C,E^c\cap C^c)
-
L_K(E,C)
-
L_K(E,E^c\cap C^c)
\\
&=
L_K(C,E^c\cap C^c)
-
L_K(E,C)
\\
&=
L_K(C,C^c\setminus E)
-
L_K(E,C)
\\
&=
L_K(C,C^c)
-
L_K(C,E)
-
L_K(E,C)
\\
&=
P_K(C)-
2L_K(E,C).
\end{split}
\end{equation}
We now apply \cref{res:schwarz} to the convex bodies~$E$ and~$C$, so that 
\begin{equation*}
L_K(E,C)\le L_K(E^{\#\nu},C^{\#\nu})
\quad
\text{and}
\quad
P_K(C)\ge P_K(C^{\#\nu}).
\end{equation*}
As a consequence, reading the chain of equalities in~\eqref{eq:insiemistica} backwards, we get that
\begin{equation}
\label{eq:deficit_symm}
\delta_K(B,A)
\ge 
\delta_K(E\cup C,E)
\ge
\delta_K(E^{\#\nu}\cup C^{\#\nu},E^{\#\nu})
=
P_K(C^{\#\nu})-
2L_K(E^{\#\nu},C^{\#\nu}).
\end{equation} 
In particular, setting $A^*=E^{\#\nu}$ and $B^*=E^{\#\nu}\cup C^{\#\nu}$, we proved that 
\begin{equation*}
\delta_K(B,A)
\ge
\delta_K(B^*,A^*),
\end{equation*}
with equality obviously if $A=A^*$ and $B=B^*$,
reducing the proof of \cref{res:quantitative} to the case of the symmetric convex bodies $A^*,B^*$.

\vspace{1ex}

\textit{Step~2: description of $C^{\#\nu}$}.  
We now note that $C^{\#\nu}$ is uniquely determined by $a-b$ and $\mathscr H^{n-1}(B\cap \partial H)$.
Indeed, since by definition 
\begin{equation*}
(x-a)\cdot\nu=0
\quad
\text{for all}\
x\in\partial H,
\end{equation*}
the half-space $H$ is preserved under Schwartz $\nu$-symmetrization, $H^{\#\nu}=H$.
In particular, we recognize that $E^{\#\nu}=B^{\#\nu}\cap H$ and hence $C^{\#\nu}
=
\mathcal C(b,E^{\#\nu}\cap \partial H)
=
\mathcal C(b,B^{\#\nu}\cap \partial H)$.
Moreover, since $a\in E^{\#\nu}$ is still the orthogonal projection of $b\in C^{\#\nu}$ on~$E^{\#\nu}$, $\partial H$ is a supporting hyperplane for $E^{\#\nu}$ in~$a$ and thus, since $E^{\#\nu}\subset H$, we conclude that
\begin{equation*}
h=h(A,B)=h(E,E\cup C)=h(E^{\#\nu},E^{\#\nu}\cup C^{\#\nu}).
\end{equation*}
Finally, by definition of Schwartz symmetrization in~\eqref{eq:schwartz}, we have
\begin{equation*}
\mathscr H^{n-1}(B\cap \partial H)
=
\mathscr H^{n-1}(B^{\#\nu}\cap \partial H)
=
\mathscr H^{n-1}(E^{\#\nu}\cap \partial H).
\end{equation*} 
In conclusion, we get that 
$C^{\#\nu}
=
\mathcal C(b,D_r^\nu(a))$,
where 
\begin{equation}
\label{eq:disco}
r=\sqrt[n-1]{\frac{\mathscr H^{n-1}(B\cap \partial H)}{\omega_{n-1}}},
\end{equation}
see \cref{fig:sym}.
Arguing similarly, also the (unbounded and closed) cone $C_\infty^{\#\nu}=\mathcal C_\infty(b,B^{\#\nu}\cap \partial H)$ is uniquely determined as $C_\infty^{\#\nu}=\mathcal C_\infty(b,D^\nu_r(a))$.
 
\vspace{1ex}

\textit{Step~3: description of $E^{\#\nu}$}.
By definition, each slice $(E^{\#\nu})^\nu_t$ with $t\in\R$ is a (possibly empty) bounded and closed $(n-1)$-dimensional disc.
More precisely, from now on assuming $(b-a)\cdot\nu\ge0$ without loss of generality, we have that 
\begin{equation*}
E^{\#\nu}
=
\bigcup_{t\in[0,d_E]}
D_{R_{E^{\#\nu}}(t)}^\nu(a-t\nu)
\end{equation*} 
for some $d_E\in[0,+\infty)$ and some concave function $R_{E^{\#\nu}}\colon[0,d_E]\to[0,+\infty)$ such that $R_{E^{\#\nu}}(0)=r$ as in~\eqref{eq:disco}.
Moreover, recalling the definition in~\eqref{eq:schwartz},
\begin{equation*}
|E^{\#\nu}|
=
|E|
=
|B\cap H|.
\end{equation*} 
Finally, by construction,  $E^{\#\nu}\subset \overline{C^{\#\nu}_\infty\setminus C^{\#\nu}}$, which equivalently rewrites as
\begin{equation*}
R_{E^{\#\nu}}(t)\le \frac rh \,t+r
\quad
\text{for}\ t\in[0,d_{E^{\#\nu}}].
\end{equation*}

\vspace{1ex} 

\textit{Step~4: construction of a family of symmetric sets}.
We now set $w=|E|$ for brevity.
We let $\mathcal F$ be the family of convex bodies $F\subset\R^n$ such that $|F|=w$ and 
\begin{equation}
\label{eq:unione_dischi}
F
=
\bigcup_{t\in[0,d_F]}
D_{R_F(t)}^\nu(a-t\nu)
\end{equation}
for some $d_F\in[0,+\infty)$ and some concave function $R_F\colon[0,d_F]\to[0,+\infty)$ with $R_F(0)=r$ as in~\eqref{eq:disco} and
\begin{equation}
\label{eq:bound_profile}
R_F(t)\le \frac rh \,t+r
\quad
\text{for}\ t\in[0,d_F].
\end{equation}
Note that $F^{\#\nu}=F$ and $F\subset \overline{C^{\#\nu}_\infty\setminus C^{\#\nu}}$ for any $F\in\mathcal F$.
We now claim that 
\begin{equation}
\label{eq:claim_max_dist}
\sup_{F\in\mathcal F} d_F
\le
\frac{nwr^{1-n}}{\omega_{n-1}}.
\end{equation} 
Indeed, given any $F\in\mathcal F$, we have  $q_F=a-d_F\nu\in F$ by~\eqref{eq:unione_dischi}. 
Consequently, since $F\in\mathcal F$ is a convex body, the (bounded and closed) convex cone $\mathcal C(q_F,D_r^\nu(a))$ is contained in $F$, so that 
$|\mathcal C(q_F,D_r^\nu(a))|\le|F|$, which means that 
\begin{equation}
\label{eq:stima_d}
\frac{\omega_{n-1}}{n}\,r^{n-1}d_F\le w,
\end{equation} 
proving the claimed~\eqref{eq:claim_max_dist}.
Combining~\eqref{eq:unione_dischi}, \eqref{eq:bound_profile} and~\eqref{eq:claim_max_dist}, we get that 
\begin{equation}
\label{eq:scatolone}
F\subset 
\overline{\mathcal C(b,D^\nu_{R_w}(q_w))\setminus C^{\#\nu}}
\end{equation}
where $d_w=\frac{nwr^{1-n}}{\omega_{n-1}}$, $q_w=a-d_w\nu$ and $R_w=\frac rh\,d_w+r$, see \cref{fig:F}.
We conclude by claiming that the family~$\mathcal F$, endowed with the Hausdorff distance $h(\cdot,\cdot)$, is a compact metric space.
Indeed, if $(F_j)_{j\in\N}\subset\mathcal F$ is such that $h(F_j,\bar F)\to0^+$ as $j\to+\infty$ for some $\bar F\subset\R^n$, then $\bar F$ is a convex body by Blaschke's Selection Theorem (see~\cite{S14}*{Th.~1.8.7} for instance).
Since $\sup_{j\in\N}P(F_j)<+\infty$ by~\eqref{eq:scatolone}, up to a subsequence, we also get that $\chi_{F_j}\to\chi_{\bar F}$ in $L^1(\R^n)$ and a.e.\ in~$\R^n$ as $j\to+\infty$ (see~\cite{M12}*{Th.~12.26} for example).
As a consequence, the limit convex body $\bar F$ still satisfies~\eqref{eq:scatolone} and $|\bar F|=w$.
Moreover, thanks to~\eqref{eq:stima_d}, up to a subsequence, we may assume that $d_{F_j}\to \bar d$ monotonically as $j\to+\infty$, for some $\bar d\in[0,d_w]$.
Therefore, recalling~\eqref{eq:nu-slice} and since 
\begin{equation*}
|F_j\bigtriangleup \bar F|
=
\int_\R
\Haus^{n-1}\big((F_j\bigtriangleup \bar F)^\nu_t\big)
\di t
=
\int_\R
\Haus^{n-1}\big((F_j)^\nu_t\bigtriangleup (\bar F)^\nu_t\big)
\di t
\end{equation*}  
for all $j\in\N$, again up to a subsequence, we get that $\Haus^{n-1}\big((F_j)^\nu_t\bigtriangleup (\bar F)^\nu_t\big)\to0^+$ as $j\to+\infty$ for a.e.\ $t\in\R$, proving that $\bar F$ satisfies~\eqref{eq:unione_dischi} for $d_{\bar F}=\bar d$ and a convex function $R_{\bar F}\colon[0,d_{\bar F}]\to [0,+\infty)$ as in~\eqref{eq:bound_profile}.
We hence get that $\bar F\in\mathcal F$, yielding the claim.

\begin{figure}
\centering
\begin{tikzpicture}[scale=.5]
\fill [teal, opacity=0.3] (6,-2.5)--(6,2.5)--(4.5,3.12) to [out=170,in=-175,distance=3.75cm] (4.5,-3.12)--(6,-2.5);
\draw [dashed] (0,-5)--(0,5)--(12,0)--(0,-5);
\draw [->,dashed] (-1,0)--(13,0);
\draw (6,-2.5)--(6,2.5)--(4.5,3.12) to [out=170,in=-175,distance=3.75cm] (4.5,-3.12)--(6,-2.5);
\draw [dashed] (3.5,-3)--(3.5,3.15);
\draw (3.5,0) node [below] {$a-t\nu$};
\draw (3.5,0) node {\tiny\textbullet};
\draw [decorate,decoration={brace,amplitude=5pt,mirror},xshift=2pt,yshift=0pt] (3.5,0) -- (3.5,3.1) node [black,midway,xshift=20pt] {$R_F(t)$};
\draw (0,0) node [below left] {$q_w$};
\draw (0,0) node {\tiny\textbullet};
\draw (6,0) node [below right] {$a$};
\draw (6,0) node {\tiny\textbullet};
\draw (12,0) node [below] {$b$};
\draw (12,0) node {\tiny\textbullet};
\draw (12.25,0) node [above right] {$\nu$};
\draw (1.8,0) node [below left] {$q_F$};
\draw (1.7,0) node {\tiny\textbullet};
\draw (5,-1.75) node {$F$};
\draw (1,6) node {$\mathcal C(b,D^\nu_{R_w}(q_w))$};
\draw (7.5,3) node {$C^{\#\nu}$};
\draw [decorate,decoration={brace,amplitude=5pt},xshift=0pt,yshift=2pt] (6.1,0) -- (11.9,0) node [black,midway,yshift=11pt] {$h$};
\draw [decorate,decoration={brace,amplitude=5pt},xshift=0pt,yshift=2pt] (0.1,0) -- (5.9,0) node [black,midway,yshift=12pt] {$d_w$};
\draw [decorate,decoration={brace,amplitude=5pt},xshift=-2pt,yshift=0pt] (0,0) -- (0,5) node [black,midway,xshift=-15pt] {$R_w$};
\begin{scope}[shift={(16.5,0)}]
\fill [teal, opacity=0.3] (6,2.5) arc(90:-90:20pt and 70pt);
\fill [teal, opacity=0.3] (6,-2.5)--(6,2.5)--(4.5,3.12) to [out=170,in=-175,distance=3.75cm] (4.5,-3.12)--(6,-2.5);
\draw (6,0) ellipse [x radius=20pt,y radius=70pt];
\draw [dashed] (4.1,3.1) arc(90:-90:20pt and 88pt);
\draw (4.1,3.1) arc(90:270:20pt and 88pt);
\draw [->,dashed] (-1,0)--(13,0);
\draw [dashed] (0,5)--(12,0)--(0,-5);
\draw [dashed] (6,0)--(6,2.5);
\draw [dashed] (4.1,0)--(4.1,3);
\draw [dashed] (0,0) ellipse [x radius=20pt,y radius=141pt];
\draw (12.25,0) node [above right] {$\nu$};
\draw (6,2.5)--(4.5,3.12) to [out=170,in=-175,distance=3.75cm] (4.5,-3.12)--(6,-2.5);
\draw (6,0) node [below] {$a$};
\draw (6,0) node {\tiny\textbullet};
\draw (12,0) node [below] {$b$};
\draw (12,0) node {\tiny\textbullet};
\draw (5.85,1) node [right] {$r$};
\draw (1.6,1) node [right] {$R_F(t)$};
\draw (1,6) node {$\mathcal C(b,D^\nu_{R_w}(q_w))$};
\draw (7.5,3) node {$C^{\#\nu}$};
\draw (6.5,-4.1) node {$D^\nu_{R_F(t)}(a-t\nu)$};
\draw (7,-3.1) node {$D^\nu_r(a)$};
\draw (4,0) node [below] {$a-t\nu$};
\draw (4,0) node {\tiny\textbullet};
\draw (5,-1.75) node {$F$};
\end{scope}
\end{tikzpicture}
\caption{An element of the family $\mathcal F$ constructed in the proof of \cref{res:quantitative}: its main lengths (left) and its circular sections orthogonal to $\nu\in\mathbb S^{n-1}$ (right).}
\label{fig:F}
\end{figure}
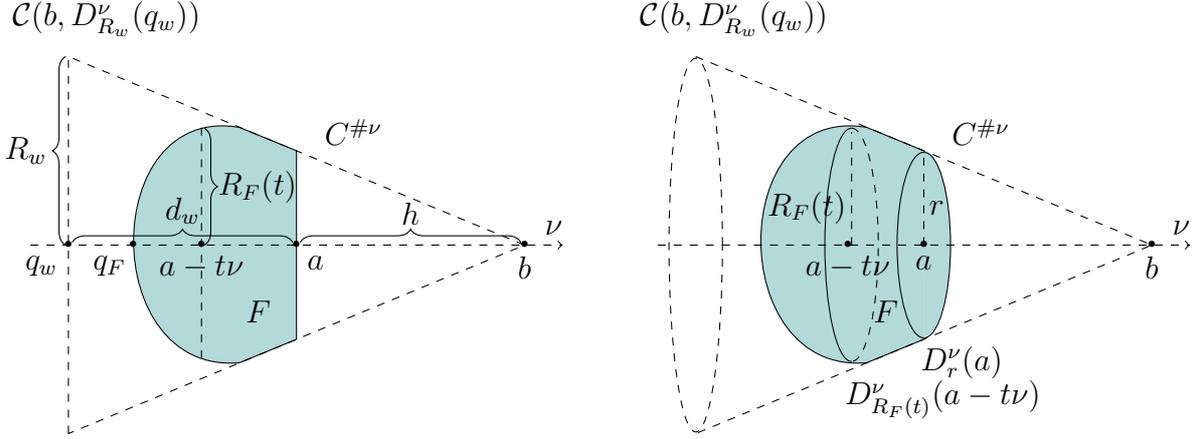

\vspace{1ex}

\textit{Step~5: definition of $f_{K,\nu}$}.
By~\eqref{eq:deficit_symm} and since $E^{\#\nu}\in\mathcal F$ by Step~3, we get that
\begin{equation}
\label{eq:scendo_al_max}
\delta_K(B,A)
\ge
P_K(C^{\#\nu})-
2L_K(E^{\#\nu},C^{\#\nu})
\ge
P_K(C^{\#\nu})-2\sup_{F\in\mathcal F}
L_K(F,C^{\#\nu}).
\end{equation}
Now consider the maximization problem
\begin{equation}
\label{eq:max_problem}
m=\sup_{F\in\mathcal F}
L_K(F,C^{\#\nu}).
\end{equation}
Since $F\subset (C^{\#\nu})^c$ for any $F\in\mathcal F$ by Step~4, we can trivially estimate
\begin{equation*}
m\le L_K((C^{\#\nu})^c,C^{\#\nu})
=
P_K(C^{\#\nu})<+\infty.
\end{equation*}
Hence let $(F_j)_{j\in\N}\subset\mathcal F$ be any sequence such that $L_K(F_j,C^{\#\nu})\to m$ as $j\to+\infty$.
By the compactness of the metric space $(\mathcal F,h)$ proved in Step~4, up to a subsequence, we find some $M\in\mathcal F$ such that $h(F_j,M)\to0^+$ as $j\to+\infty$.
Since also $\sup_{j\in\N}P(F_j)<+\infty$ by~\eqref{eq:scatolone}, up to a subsequence, we also get that $\chi_{F_j}\to\chi_M$ in $L^1(\R^n)$ and a.e.\ in~$\R^n$ as $j\to+\infty$.
By Fatou's Lemma, we hence infer that 
\begin{equation}
\label{eq:realizzo_m}
m\le 
L_K(M,C^{\#\nu})
\le
\liminf_{j\to+\infty}
L_K(F_j,C^{\#\nu})
=
m,
\end{equation}
yielding that $M\in\mathcal F$ is a maximizer of~\eqref{eq:max_problem}.
Now we observe that 
\begin{equation*}
P_K(C^{\#\nu})-2L_K(M,C^{\#\nu})\ge0.
\end{equation*}
Indeed, arguing exactly as in~\eqref{eq:insiemistica}, we check that
\begin{equation}
\label{eq:non-neg}
P_K(C^{\#\nu})-2L_K(M,C^{\#\nu})
=
P_K(M\cup C^{\#\nu})-P(M),
\end{equation}
which is non-negative by \cref{res:monotonicity}, since $M\cup C^{\#\nu}$ is a convex body by construction of the family $\mathcal F$.
Now, again by the definition of $\mathcal F$ in Step~4, the value $m=m(h,r,w)$ of the maximization problem~\eqref{eq:max_problem} is uniquely determined in terms of $h$, $r$ and $w$.
In addition, thanks to Step~3, $P_K(C^{\#\nu})$ is uniquely determined by~$h$ and~$r$. 
We hence define
\begin{equation}
\label{eq:definisco_f}
f_{K,\nu}(h,\omega_{n-1}r^{n-1},w)
=
P_K(C^{\#\nu})-2\,m(h,r,w)
=
P_K(C^{\#\nu})-2L_K(M,C^{\#\nu}),
\end{equation}
yielding~\eqref{eq:quantitative} thanks to~\eqref{eq:scendo_al_max}, \eqref{eq:max_problem} and~\eqref{eq:realizzo_m}.
Finally, in virtue of the above construction and of~\eqref{eq:non-neg} and~\eqref{eq:definisco_f}, 
the equality in~\eqref{eq:quantitative} is attained by the sets $A=M$ and $B=M\cup C^{\#\nu}$, where $M$ is any solution of the maximization problem~\eqref{eq:max_problem}.
\end{proof}

\begin{remark}[On the maximization problem~\eqref{eq:max_problem}]
With the same notation of the proof of \cref{res:quantitative}, we can rewrite 
\begin{equation}
\label{eq:potential}
L_K(F,C^{\#\nu})
=
\int_{F}g^K_{C^{\#\nu}}(x)\di x,
\quad
F\in\mathcal F,
\end{equation}
where $g^K_{C^{\#\nu}}\colon\R^n\to[0,+\infty]$ is  given by
\begin{equation*}
g^K_{C^{\#\nu}}(x)
=
(K*\chi_{C^{\#\nu}})(x)
=
\int_{\R^n}K(x-y)\,\chi_{C^{\#\nu}}(y)\di y
\quad
\text{for}\ x\in\R^n.
\end{equation*}
Hence problem~\eqref{eq:max_problem} can be equivalently interpreted as the maximization of the \emph{$g^K_{C^{\#\nu}}$-potential energy}~\eqref{eq:potential} among convex bodies  $F\in\mathcal M$.
\end{remark}

\subsection{Proof of \texorpdfstring{\cref{res:phi-quantitative,res:fractional}}{the radially-symmetric decreasing and fractional cases}}

The following result is a simple interpolation estimate:
the first part of the statement follows from~\cite{BP19}*{Prop.~2.2}, while the second part is an easy refinement.
We leave the plain details to the reader.   

\begin{lemma}[Interpolation]
\label{res:interpolation}
If $E\subset\R^n$ is a convex body, then
\begin{equation}
\label{eq:phi-interpolation}
P_K(E)
\le 
\max\set*{\tfrac12\,P(E),|E|}\,\int_{\R^n}\min\set*{1,|x|}\,K(x)\di x.
\end{equation}
In particular, in the fractional case $s\in(0,1)$, 
\begin{equation}
\label{eq:s-interpolation}
P_s(E)
\le 
\frac{2^{1-s}\,n\,\omega_n}{s\,(1-s)}
\,P(E)^s\,|E|^{1-s}.
\end{equation}
\end{lemma}

\begin{proof}[Proof of \cref{res:phi-quantitative}]
Since $P_\phi$ is invariant by rotations, we can apply \cref{res:quantitative} for $\nu=\mathrm e_n$, so that $C^{\#\mathrm e_n}=\mathcal C_r^h$ by Step~2 of the proof of \cref{res:quantitative}. 
By Steps~1 and~3 of the proof of \cref{res:quantitative}, and thanks to~\eqref{eq:phi-interpolation} in \cref{res:interpolation}, we can estimate
\begin{equation*}
\begin{split}
L_K(E^{\#\mathrm e_n},\mathcal C_r^h)
&\le 
L_K(E^{\#\mathrm e_n},(E^{\#\mathrm e_n})^c)
=
P_K(E^{\#\mathrm e_n})
\\
&\le 
\max\set*{\tfrac12\,P(E^{\#\mathrm e_n}),|E^{\#\mathrm e_n}|}\,\int_{\R^n}\min\set*{1,|x|}\,\phi(|x|)\di x.
\end{split}
\end{equation*}
The conclusion thus follows by observing that $|E^{\#\mathrm e_n}|=|B\cap H|$ thanks to the definition in~\eqref{eq:schwartz} and also that $P(E^{\#\mathrm e_n})\le P(B\cap H)$ by the Schwartz inequality for the Euclidean perimeter (see~\cite{M12}*{Th.~19.11}).
\end{proof}

\begin{proof}[Proof of \cref{res:fractional}]
We argue as in the proof of \cref{res:phi-quantitative}.
We just notice that
\begin{equation*}
P_s(C^{\#\mathrm e_n})
= 
P_s(\mathcal C_h^r)
\ge 
c_{n,s}^{\rm iso}\,|\mathcal C_r^h|^{\frac{n-s}n}
= 
c^{\rm iso}_{n,s}
\left(\tfrac{\omega_{n-1}}{n}\,h\,r^{n-1}
\right)^{1-\frac sn}
\end{equation*}
thanks to~\eqref{eq:s-iso} and that
\begin{equation*}
L_s(E^{\#\mathrm e_n},\mathcal C_r^h)
\le
P_s(E^{\#\mathrm e_n})
\le 
\frac{2^{1-s}\,n\,\omega_n}{s\,(1-s)}
\,P(E^{\#\mathrm e_n})^s\,|E^{\#\mathrm e_n}|^{1-s}
\le
\frac{2^{1-s}\,n\,\omega_n}{s\,(1-s)}
\,P(E)^s\,|E|^{1-s}
\end{equation*}
by~\eqref{eq:s-interpolation} in \cref{res:interpolation}, readily yielding the conclusion.
\end{proof}

\subsection{Proof of \texorpdfstring{\cref{res:segmenti}}{1D monotonicity}}

We conclude by dealing with the case $n=1$.

\begin{proof}[Proof of \cref{res:segmenti}]
Since $P_\phi$ is invariant by translation, we can assume that $A=(0,|A|)$ and $B=(0,|B|)$ in~$\R$. 
We now observe that, by a simple change of variables, 
\begin{equation*}
\begin{split}
P_\phi(A)
=
\int_{(0,|A|)}\int_{(0,|A|)^c}
\phi(|x-y|)\di x\di y
=
|A|^2
\int_{(0,1)}\int_{(0,1)^c}
\phi(|A|\,|\xi-\eta|)\di\xi\di\eta.
\end{split}
\end{equation*}
Therefore, by~\eqref{eq:2-decreasing}, we can estimate
\begin{equation*}
\begin{split}
P_\phi(B)-P_\phi(A)
&=
\int_{(0,1)}\int_{(0,1)^c}
\left(|B|^2\phi(|B|\,|\xi-\eta|)-|A|^2\phi(|A|\,|\xi-\eta|)\right)
\di\xi\di\eta
\\
&\ge
\int_{(0,1)}\int_{(0,1)^c}
\big(\psi(|B|)-\psi(|A|)\big)\phi(|\xi-\eta|)
\di\xi\di\eta
\\
&=\big(\psi(|B|)-\psi(|A|)\big)\,P_\phi((0,1)),
\end{split}
\end{equation*}
yielding the conclusion.
\end{proof}


\end{document}